\documentclass[12pt, reqno, a4paper]{amsart}
\usepackage{txfonts}	
\usepackage{amsmath, amssymb, amsthm, amscd}
\usepackage[T1]{fontenc}
\usepackage{eucal}
\usepackage[dvips]{color}
\usepackage{multicol}
\usepackage[all]{xy}		
\usepackage{graphicx}
\usepackage{color}
\usepackage{colordvi}
\usepackage{xspace}
\usepackage{tikz}
\usepackage{enumerate}
\usepackage{ulem} 
\usepackage[colorlinks,final,backref=page,hyperindex]{hyperref}
\usepackage{bm}
\allowdisplaybreaks % page break for formulars

\newcommand {\emptycomment}[1]{} %to remove paragraphs

\newcommand{\delete}[1]{}
\newif\ifflabel\flabelfalse
\ifflabel
\else
	
\fi
\newtheorem{theorem}{Theorem}[section]
\newtheorem{lemma}[theorem]{Lemma}
\newtheorem{corollary}[theorem]{Corollary}
\newtheorem{proposition}[theorem]{Proposition}
\theoremstyle{definition}
\newtheorem{definition}[theorem]{Definition}
\newtheorem{example}[theorem]{Example}
\newtheorem{remark}[theorem]{Remark}

\newcommand{\End}{\mathrm{End}}
\newcommand{\Hom}{\mathrm{Hom}}
\newcommand{\id}{\mathrm{id}}

\newcommand{\ad}{\mathrm{ad}}

\begin{document}

\title[Quasi-triangular and factorizable perm bialgebras]{Quasi-triangular and factorizable perm bialgebras}

\author[Y.~Lin]{Yuanchang Lin}
\address{School of Mathematics, North University of China, Taiyuan 030051, China}
\email{linyuanchang@nuc.edu.cn}



\date{\today}

\begin{abstract}
In this paper, we introduce the notions of quasi-triangular and factorizable perm bialgebras, based on notions of the perm Yang-Baxter equation and $(R, \ad)$-invariant condition. 
A factorizable perm bialgebra induces a factorization of the underlying perm algebra and the double of a perm bialgebra naturally admits a factorizable perm bialgebra structure. 
The notion of relative Rota-Baxter operators of weights on perm algebras is introduced to characterize solutions of the perm Yang-Baxter equation, whose skew-symmetric parts are $(R, \ad)$-invariant.
These operators are in one-to-one correspondence with linear transformations fulfilling a Rota-Baxter-type identity in the case of quadratic perm algebras.
Furthermore, we introduce the notion of quadratic Rota-Baxter perm algebras of weights, demonstrate that a quadratic Rota-Baxter perm algebra of weight $0$ induces a triangular perm bialgebra, and establish a one-to-one correspondence between quadratic Rota-Baxter perm algebras of nonzero weights and factorizable perm bialgebras. 
\end{abstract}

\maketitle

\tableofcontents

%%%%%%%%%%%%%%%%%%%%%%%%%%%%%%%%%%%%%%%%%%%%%%%%%%%%%%%%%%%%%%%%%%%%%%%%%%%%%%%%
%%%%%%%%%%%%%%%%%%%%%%%%%%%%%%%%%%%%%%%%%%%%%%%%%%%%%%%%%%%%%%%%%%%%%%%%%%%%%%%%
%%%%%%%%%%%%%%%%%%%%%%%%%%%%%%%%%%%%%%%%%%%%%%%%%%%%%%%%%%%%%%%%%%%%%%%%%%%%%%%%
\section{Introduction}

The purpose of this paper is to establish quasi-triangular and factorizable theories for perm bialgebras.

A bialgebra structure consists of an algebra structure and a coalgebra structure coupled by certain compatibility conditions. 
Such structures play important roles in many areas and have connections with other structures arising from mathematics and physics. 
For example, Lie bialgebras, introduced by Drinfeld \cite{drinfeld1983hamiltonian}, are the algebra structures of Poisson-Lie groups, and play a fundamental role as the infinitesimalisation of quantum groups \cite{chari1995guide, kassel2012quantum}. 
For associative algebras, the notion of infinitesimal bialgebra was introduced by Joni and Rota to provide an algebraic framework for the calculus of divided differences \cite{joni1979coalgebras}. 
Variants of this notion, such as antisymmetric infinitesimal bialgebras (also termed balanced infinitesimal bialgebras in the sense of the opposite algebras or associative D-bialgebras in the sense of Drinfeld), have been systematically studied having their close analogy with Lie bialgebras,
and are widely applied in 2d topological field and string theory \cite{aguiar2000infinitesimal, aguiar2001associative, bai2010double, joni1979coalgebras, kock2004frobenius,lauda2008open, zhelyabin1997jordan}.
Other known bialgebra structures include pre-Lie bialgebras, Leibniz bialgebras and perm bialgebras \cite{bai2008left, tang2022leibniz, lin2025infinite}.

Meanwhile, in the realm of Lie bialgebras, quasi-triangular Lie bialgebra structures have been pivotal objects in mathematical physics \cite{drinfeld1986quantum, semenov1983classical}. 
Among these quasi-triangular Lie bialgebra structures, factorizable Lie bialgebras constitute a particularly important subclass, linking classical 
$r$-matrices to factorization problems and playing a key role in integrable systems \cite{bai2010nonabelian, reshetikhin1988quantum, semenov2003integrable}. Recently, quasi-triangular and factorizable theories have been extended to antisymmetric infinitesimal bialgebras \cite{sheng2023quasi}, pre-Lie bialgebras \cite{wang2024quasi} and Leibniz bialgebras \cite{bai2024quasi}.

A perm algebra is an associative algebra with the left-commutative identity \cite{chapoton2002endofoncteur} and a unified bialgebra theory for perm algebras, called perm bialgebras, was developed in \cite{lin2025infinite}. 
The notion of the perm Yang-Baxter equation was introduced, whose symmetric solutions give rise to perm bialgebras, and $\mathcal{O}$-operators were introduced as the operator forms of symmetric solutions of the perm Yang-Baxter equation \cite{lin2025infinite}.
However, the corresponding quasi-triangular and factorizable theories for perm bialgebras remain to be discovered.

In this paper, we demonstrate the feasibility of quasi-triangular and factorizable theories in the context of perm bialgebras, synthesizing concepts from quasi-triangular Lie bialgebras, antisymmetric infinitesimal bialgebras and Leibniz bialgebras.
More precisely, we introduce the notion of quasi-triangular perm bialgebras based on the perm Yang-Baxter equation and $(R, \ad)$-invariant condition.
In particular, if the skew-symmetric part of the solution of the perm Yang-Baxter equation in a quasi-triangular perm bialgebra is nondegenerate, then a factorizable perm bialgebra is obtained.
We establish that every factorizable perm bialgebra induces a factorization of its underlying perm algebra and the double of a perm bialgebra is naturally endowed with a canonical factorizable perm bialgebra structure.
Furthermore, the notion of relative Rota-Baxter operators of weights on perm algebras is introduced to characterize solutions of the perm Yang-Baxter equation, whose skew-symmetric parts are $(R, \ad)$-invariant.

Recent work has demonstrated that factorizable Lie (resp. antisymmetric infinitesimal, Leibniz, pre-Lie) bialgebras could be characterized by quadratic Rota-Baxter Lie (resp. Frobenius, Leibniz, pre-Lie) algebra of nonzero weight \cite{lang2023factorizable, sheng2023quasi, bai2024quasi, wang2024quasi}.
This motivates our investigation of analogous Rota-Baxter characterizations of factorizable perm bialgebras.
For this purpose, we introduce the notion of quadratic Rota-Baxter perm algebras by equipping quadratic perm algebras with Rota-Baxter operators satisfying a compatibility condition, and establish a one-to-one correspondence between factorizable perm bialgebras and quadratic Rota-Baxter perm algebras of nonzero weight.

%%%%%%%%%%%%%%%%%%%%%%%%%%%%%%%%%%%%%%%%%%%%%%%%%%%%%%%%%%%%%%%%%%%%%%%%%%%%%%%%
The paper is organized as follows.
In Section~\ref{sec:qtpb}, we introduce the notion of quasi-triangular perm bialgebras as a special subclass of perm bialgebras.
Section~\ref{sec:fpb} presents the concept of factorizable perm bialgebras, which is a distinguished subclass of quasi-triangular perm  bialgebras.
We demonstrate that a factorizable perm bialgebra induces a factorization of the underlying perm algebra.
Furthermore, we prove that the double of a perm bialgebra is naturally endowed with a factorizable perm bialgebra structure.
Section~\ref{sec:rboqrpa} establishes the operator forms of solutions of the perm Yang-Baxter equation under the $(R, \ad)$-invariant condition.
We demonstrate that relative Rota-Baxter operators of weights on perm algebras could be used to characterize solutions of the perm Yang-Baxter equation, whose skew-symmetric parts are $(R, \ad)$-invariant, and show that these operators are in one-to-one correspondence with linear transformations satisfying a Rota-Baxter-type identity in the case of quadratic perm algebras.
In Section~\ref{sec:qrflb}, we demonstrate the Rota-Baxter characterization of factorizable perm bialgebras.
We introduce the notion of quadratic Rota-Baxter perm algebras of weights, demonstrate that a quadratic Rota-Baxter perm algebra of weight $0$ induces a triangular perm bialgebra, and establish a one-to-one correspondence between quadratic Rota-Baxter perm algebras of nonzero weights and factorizable perm bialgebras.

Throughout this paper, we work over a base field $\mathbb{K}$ of characteristic zero, and all vector spaces and algebras are assumed to be finite-dimensional. 
We adopt the following conventions and notation.
\begin{enumerate}
	\item 
	Let $(A, \diamond)$ be a vector space equipped with a bilinear operation $\diamond: A \otimes A \to A$.
	Let $L_\diamond(a)$ and $R_\diamond(a)$ denote the left and right multiplication operator respectively, that is 
	\begin{equation*}
		L_\diamond(a)b = R_\diamond(b)a = a \diamond b, \;\; \forall a, b \in A.
	\end{equation*}
	In particular, $\ad_\diamond(a)$ is defined by $\ad_\diamond(a) = (L_\diamond - R_\diamond)(a)$.
	We also simply denote them by $L(a)$, $R(a)$ and $\ad(a)$ respectively without confusion.

	\item 
	Let $V$ be a vector space. 
	Let $\sigma: V \otimes V \to V \otimes V$ be the twisting operator defined by
	\begin{equation*}
		\sigma(u \otimes v) = v \otimes u, \;\; \forall u, v \in V.
	\end{equation*}

	\item 
	Let $V, W$ be two vector spaces and $T: V \to W$ be a linear map.
	Denote the dual map by $T^*: W^* \to V^*$, which is defined by 
	\begin{equation*}
		\langle v, T^*(w^*)\rangle = \langle T(v), w^*\rangle, \;\; \forall v \in V, w^* \in W^*.
	\end{equation*}

	\item 
	Let $A,V$ be a vector space.
	For a linear map $\mu: A \to \End(V)$, define a linear map $\mu^*: A \to \End(V^*)$ by 
	\begin{equation*}
		\langle \mu^*(a)v^*, v\rangle = \langle v^*, \mu(a) v\rangle, \;\; \forall a \in A, v \in V, v^* \in V^*,
	\end{equation*}
	that is, $\mu^*(a) = (\mu(a))^*$.
	
\end{enumerate}

%%%%%%%%%%%%%%%%%%%%%%%%%%%%%%%%%%%%%%%%%%%%%%%%%%%%%%%%%%%%%%%%%%%%%%%%%%%%%%%%
%%%%%%%%%%%%%%%%%%%%%%%%%%%%%%%%%%%%%%%%%%%%%%%%%%%%%%%%%%%%%%%%%%%%%%%%%%%%%%%%
%%%%%%%%%%%%%%%%%%%%%%%%%%%%%%%%%%%%%%%%%%%%%%%%%%%%%%%%%%%%%%%%%%%%%%%%%%%%%%%%
\section{Quasi-triangular perm bialgebras}\label{sec:qtpb}
In this section, we first recall the notion of perm bialgebras and then introduce the notion of quasi-triangular perm bialgebras as a special subclass of perm bialgebras, based on notions of the perm Yang-Baxter equation and $(R, \ad)$-invariant condition.
%%%%%%%%%%%%%%%%%%%%%%%%%%%%%%%%%%%%%%%%%%%%%%%%%%%%%%%%%%%%%%%%%%%%%%%%%%%%%%%%

\begin{definition}
	A \textbf{perm algebra} $(A, \cdot_{A})$ is a vector space $A$ with a multiplication $\cdot_A: A \otimes A \to A$ such that
	\begin{equation*}
		 a \cdot_A (b \cdot_A c)= (a \cdot_A b) \cdot_A c = (b \cdot_A a) \cdot_A c, \;\; \forall a, b, c \in A. 
	\end{equation*}
\end{definition}

\begin{definition}
	A \textbf{perm coalgebra} $(A, \Delta)$ is a vector space $A$ with a comultiplication $\Delta: A \to A \otimes A$ such that
	\begin{equation*}
		(\Delta \otimes \id)\Delta = (\id \otimes \Delta) \Delta = (\sigma \otimes \id)(\Delta \otimes \id)\Delta.
	\end{equation*}
	A \textbf{perm bialgebra} is a triple $(A, \cdot_{A}, \Delta)$ where $(A, \cdot_A)$ is a perm algebra and $(A, \Delta)$ is a perm coalgebra such that the following equations hold:
	\begin{align*}
		\Delta(a \cdot_A b) &= (\ad(a) \otimes \id)\Delta(b) + (\id \otimes R(b))\Delta(a), \\
		\sigma (R(b) \otimes \id)\Delta(a) &= (R(a) \otimes \id)\Delta(b), \\
		\Delta(a \cdot_A b) &= (\id \otimes L(a))\Delta(b) + (\ad(b) \otimes \id)(\Delta(a)-\sigma \Delta(a)), \;\; \forall a, b \in A.
	\end{align*}
\end{definition}

\begin{remark}
	Note that $(A, \Delta)$ is a perm coalgebra if and only if $(A^*, \cdot_{A^*})$ is a perm algebra, 
	where the multiplication $\cdot_{A^*}: A^* \otimes A^* \to A^*$ on $A^*$ is defined as the linear dual of $\Delta$ via
	\begin{equation}
		\langle \Delta(a), a^* \otimes b^*\rangle = \langle a, a^* \cdot_{A^*} b^*\rangle, 
		\;\; \forall a \in A, \; a^*, b^* \in A^*. \label{eq:cmul}
	\end{equation}
	Consequently, a perm bialgebra $(A, \cdot_{A}, \Delta)$ can alternatively be denoted by $((A, \cdot_{A}), (A^*, \cdot_{A^*}))$, where $(A^*, \cdot_{A^*})$ is the perm algebra structure on $A^*$ induced by the perm coalgebra $(A, \Delta)$ through Eq.~\eqref{eq:cmul}.
\end{remark}

Let $r = \sum_{i} a_i \otimes b_i \in A
\otimes A$ and $r^\prime = \sum_{j} a_j^\prime
\otimes b_j^\prime \in A \otimes A$.
Set
\begin{eqnarray*}
	&&r_{12} \cdot r_{23}^\prime := \sum_{i, j} a_i \otimes b_i \cdot a_j^\prime \otimes b_j^\prime, \;\;\; r_{13} \cdot r_{23}^\prime := \sum_{i, j} a_i \otimes a_j^\prime \otimes b_i \cdot b_j^\prime, \\
	&&r_{13} \cdot r_{12}^\prime := \sum_{i, j} a_i \cdot a_j^\prime \otimes b_j^\prime \otimes b_i, \;\;\; r_{12} \cdot r_{13}^\prime := \sum_{i, j} a_i \cdot a_j^\prime \otimes b_i \otimes b_j^\prime, \\
	&&r_{23} \cdot r_{13}^\prime := \sum_{i, j} a_j^\prime \otimes a_i \otimes b_i \cdot b_j^\prime, \;\;\; r_{23} \cdot r_{12}^\prime := \sum_{i, j} a_j^\prime \otimes a_i \cdot b_j^\prime \otimes b_i.
\end{eqnarray*}

\begin{lemma}\label{lem:cbd}
	{\rm (\cite{lin2025infinite})}
	Let $(A, \cdot_A)$ be a perm algebra and $r \in A \otimes A$.
	Define a comultiplication $\Delta_r: A \to A \otimes A$ by
	\begin{equation}
		\Delta_r(a) = (\id \otimes R(a) - \ad(a) \otimes \id)(r), \;\; \forall a \in A. \label{eq:pcbd}
	\end{equation}
	Then $(A, \cdot_A, \Delta_r)$ is a perm bialgebra if and only the following equations hold:
	\begin{align*}
		0 &= (R(a) \otimes R(b))(r-\sigma(r)) , \\
		0 &= ( (R(b \cdot a)-R(a \cdot b)) \otimes \id - \ad(b) \otimes \ad(a) )(r-\sigma(r)) , \\
		0 &= (\id \otimes \id \otimes R(a) - (\ad(a) \otimes \id \otimes \id)(\sigma \otimes \id))(r_{12} \cdot r_{23} - r_{13} \cdot r_{23} + r_{12} \cdot r_{13} - r_{13} \cdot r_{12} )  \nonumber \\
		&\;\;\;\;- (\ad(a) \otimes \id \otimes \id)(\sigma \otimes \id) ( (\sigma(r)-r)_{12} \cdot r_{23} + (\sigma(r)-r)_{12} \cdot r_{13} - r_{13} \cdot (\sigma(r)-r)_{12}),  \\ 
		0 & = (\id \otimes \id \otimes R(a))(\id \otimes \id \otimes \id - \sigma \otimes \id )(r_{12} \cdot r_{23} - r_{13} \cdot r_{23} + r_{12}\cdot r_{13} - r_{13} \cdot r_{12}  ) \nonumber \\
		&\;\;\;\;-(\id \otimes \ad(a) \otimes \id)( (r - \sigma(r))_{12} \cdot r_{23}) - (\ad(a) \otimes \id \otimes \id)((r-\sigma(r))_{12} \cdot r_{13}), 
	\end{align*}
	for all $a, b \in A$.
\end{lemma}

\begin{definition}
	Let $(A, \cdot_A)$ be a perm algebra and $r \in A \otimes A$.
	The equation
	\begin{equation*}
		[[r, r]] := r_{12} \cdot_A r_{23} - r_{13} \cdot_A r_{23} + r_{12} \cdot_A r_{13} - r_{13} \cdot_A r_{12}  = 0
	\end{equation*}
	is called the \textbf{perm Yang-Baxter equation} in $(A, \cdot_A)$.
\end{definition}

\begin{definition}
	Let $(A, \cdot_A)$ be a perm algebra and $r \in A \otimes A$.
	Then $r$ is called \textbf{$(R, \ad)$-invariant} if
	\begin{equation*}
		(\id \otimes R(a) - \ad(a) \otimes \id)(r) = 0, \;\; \forall a \in A.
	\end{equation*}
\end{definition}

\begin{theorem}\label{thm:cbd}
	Let $(A, \cdot_A)$ be a perm algebra and $r \in A \otimes A$.
	Define a comultiplication $\Delta_r: A \to A \otimes A$ by Eq.~\eqref{eq:pcbd}.
	If $r$ is a solution of the perm Yang-Baxter equation in $(A, \cdot_A)$ and the skew-symmetric part $\frac{1}{2}(r - \sigma(r))$ is $(R, \ad)$-invariant, namely, 
	\begin{equation}
		(\id \otimes R(a) - \ad(a) \otimes \id)(r - \sigma(r)) = 0, \;\; \forall a \in A,\label{eq:rsrradi}
	\end{equation}
	then $(A, \cdot, \Delta_r)$ is a perm bialgebra.
\end{theorem}
\begin{proof}
	Let $a, b \in A$.
	Note that 
	\begin{align*}
		&(R(a) \otimes R(b))(r-\sigma(r)) = (R(a) \otimes \id)(\id \otimes R(b) - \ad(b) \otimes \id)(r-\sigma(r)), \\
		&( (R(b \cdot a)-R(a \cdot b)) \otimes \id - \ad(b) \otimes \ad(a) )(r-\sigma(r)) \\
		&= \sigma(\id \otimes \ad(b))(\ad(a) \otimes \id - \id \otimes R(a))(r - \sigma(r)), \\
		&(\ad(a) \otimes \id \otimes \id)(\sigma \otimes \id) ( (\sigma(r)-r)_{12} \cdot r_{23} + (\sigma(r)-r)_{12} \cdot r_{13} - r_{13} \cdot (\sigma(r)-r)_{12}) \\
		&= -\sum_{j} (\ad(a) \otimes \id \otimes \id) (\sigma \otimes \id)( (\id \otimes R(a_j) - \ad(a_j) \otimes \id)(r - \sigma(r)) \otimes b_j), \\
		&(\id \otimes \ad(a) \otimes \id)( (r - \sigma(r))_{12} \cdot r_{23}) + (\ad(a) \otimes \id \otimes \id)((r-\sigma(r))_{12} \cdot r_{13}) \\
		&=\sum_{j} (\id \otimes R(a \cdot_{A} a_j - a_j \cdot_{A} a) - \ad(a \cdot_{A} a_j - a_j \cdot_{A} a) \otimes \id)(r - \sigma(r)) \otimes b_j.
	\end{align*}
	The theorem follows immediately from Lemma~\ref{lem:cbd}.
\end{proof}

Motivated by Theorem~\ref{thm:cbd}, we introduce the following notion of a quasi-triangular perm bialgebra as a distinguished subclass of perm bialgebras.
\begin{definition}
	Let $(A, \cdot_A)$ be a perm algebra. 
	If $r \in A \otimes A$ is a solution of the perm Yang-Baxter equation in $(A, \cdot_{A})$ and the skew-symmetric part of $r$ is $(R, \ad)$-invariant, then the perm bialgebra $(A, \cdot_A, \Delta_r)$ with $\Delta_r$ defined by Eq.~\eqref{eq:pcbd} via $r$ is called a \textbf{quasi-triangular perm bialgebra}.
	In particular, if $r$ is symmetric, then $(A, \cdot_A, \Delta_r)$ is called a \textbf{triangular perm bialgebra}.
\end{definition}

The following result demonstrates that quasi-triangular perm bialgebras always come in pairs.
\begin{proposition}\label{prop:rsrqt}
	Let $(A, \cdot_{A})$ be a perm algebra and $r \in A \otimes A$.
	If the skew-symmetric part of $r$ is $(R, \ad)$-invariant, then $r$ is a solution of the perm Yang-Baxter equation in $(A, \cdot_{A})$ if and only if $\sigma(r)$ is a solution of the perm Yang-Baxter equation in $(A, \cdot_{A})$.
	Moreover, if $(A, \cdot_{A}, \Delta_r)$ is a quasi-triangular perm bialgebra, then $(A, \cdot_{A}, \Delta_{\sigma(r)})$ is also a quasi-triangular perm bialgebra.
\end{proposition}
\begin{proof}
	Set $r = \sum_{i} a_i \otimes b_i$.
	Then 
	\begin{align*}
		&[[\sigma(r), \sigma(r)]] \\
		&=\sum_{i,j} ( b_i \otimes a_i \cdot_{A} b_j \otimes a_j - b_i \otimes b_j \otimes a_i \cdot_{A} a_j + b_i \cdot_{A} b_j \otimes a_i \otimes a_j - b_j \cdot_{A} b_i \otimes a_i \otimes a_j ) \\
		&= \sum_{j} (\id \otimes R(b_j) - \ad(b_j) \otimes \id)(\sigma(r)) \otimes a_j - \sum_{i, j} b_i \otimes b_j \otimes a_i \cdot_{A} a_j \\
		&= \sum_{j} (\id \otimes R(b_j) - \ad(b_j) \otimes \id)(\sigma(r) - r) \otimes a_j \\
		&\quad + \sum_{i,j} (a_i \otimes b_i \cdot_{A} b_j \otimes a_j - b_j \cdot_{A} a_i \otimes b_i \otimes a_j + a_i \cdot_{A} b_j \otimes b_i \otimes a_j - b_i \otimes b_j \otimes a_i \cdot_{A} a_j) \\
		&= \sum_{j} (\id \otimes R(b_j) - \ad(b_j) \otimes \id)(\sigma(r) - r) \otimes a_j \\
		&\quad + (\id \otimes \sigma)\sum_{i,j} a_i \otimes a_j \otimes b_i \cdot_{A} b_j + (\id \otimes \sigma)\sum_{i, j} (-\id \otimes R(a_j) + \ad(a_j) \otimes \id)(\sigma(r))  \\
		&= \sum_{j} (\id \otimes R(b_j) - \ad(b_j) \otimes \id)(\sigma(r) - r) \otimes a_j \\
		&\quad + (\id \otimes \sigma)( \sum_{j} (\id \otimes R(a_j) - \ad(a_j) \otimes \id)(r - \sigma(r)) \otimes b_j) - (\id \otimes \sigma)[[r, r]].
	\end{align*}
	By the assumption, we show that $[[r, r]] = 0$ if and only if $[[\sigma(r), \sigma(r)]] = 0$.
	The remainder of the argument is obvious.
\end{proof}

For a vector space A, through the isomorphism $A \otimes A \cong \Hom(A^*, A)$, any $r = \sum_i a_i \otimes b_i \in A \otimes A$ can be identified as a map from $A^*$ to $A$, which we denote by $T_{r}: A^* \to A$, explicitly,
\begin{equation*}
	T_r: A^* \to A, \; a^* \mapsto \sum_i \langle a^*, a_i \rangle b_i, \;\; \forall a^* \in A.
\end{equation*}
That is, $\langle T_{r}(a^*), b^*\rangle = \langle r, a^* \otimes b^*\rangle$ for all $a^*, b^* \in A^*$.

\begin{proposition}\label{prop:pybeof}
	Let $(A, \cdot_A)$ be a perm algebra and $r \in A \otimes A$.
	Then $r$ is a solution of the perm Yang-Baxter equation in $(A, \cdot_A)$ if and only if the following equation holds:
	\begin{equation}
		T_{r}(a^*) \cdot_A T_{r}(b^*) = T_{r}( L^*(T_{r}(a^*)) b^* + \ad^*(T_{\sigma(r)}(b^*)) a^* ), \;\; \forall a^*, b^* \in A^*. \label{eq:pybeof}
	\end{equation}
	Furthermore, if the skew-symmetric part of $r$ is $(R, \ad)$-invariant, then $r$ is a solution of the perm Yang-Baxter equation in $(A, \cdot)$ if and only if the following equation holds:
	\begin{equation}
		T_{\sigma(r)}(a^*) \cdot_A T_{\sigma(r)}(b^*) = T_{\sigma(r)}(L^*(T_{r}(a^*)) b^* + \ad^*(T_{\sigma(r)}(b^*)) a^*), \;\; \forall a^*, b^* \in A^*. \label{eq:pybeofs}
	\end{equation}
\end{proposition}
\begin{proof}
	Set $r = \sum_{i} a_i \otimes b_i$.
	For all $a^*, b^*, c^* \in A^*$, we have
	\begin{align*}
		&\langle T_{r}(a^*) \cdot_A T_{r}(b^*), c^*\rangle \!=\! \langle \sum_{i,j}\langle a^*, a_i \rangle \langle b^*, a_j \rangle b_i \cdot_A b_j, c^* \rangle \!=\! \sum_{i, j}\langle a_i \otimes a_j \otimes b_i \cdot_{A} b_j, a^* \otimes b^* \otimes c^* \rangle, \\
		&\langle T_{r}( L^*(T_{r}(a^*)) b^*), c^*\rangle \!=\! \sum_{i,j}\langle a^*, a_i \rangle\langle L^*(b_i) b^*, a_j\rangle\langle  b_j, c^* \rangle \!=\! \sum_{i,j}\langle a_i \otimes b_i \cdot_{A} a_j \otimes b_j, a^* \otimes b^* \otimes c^* \rangle, \\
		&\langle T_{r}( \ad^*(T_{\sigma(r)}(b^*)) a^* ), c^*\rangle \!=\! \sum_{i, j}\langle b^*, b_i \rangle \langle \ad^*(a_i)a^*, a_j\rangle \langle b_j, c^*\rangle \!=\! \sum_{i, j}\langle \ad(a_i) a_j \otimes b_i \otimes b_j, a^* \otimes b^* \otimes c^*\rangle.
	\end{align*}
	Thus, $r$ is a solution of the perm Yang-Baxter equation if and only if Eq.~\eqref{eq:pybeof} holds.
	
	Similarly, we have 
	\begin{align*}
		&\langle T_{\sigma(r)}(a^*) \cdot_A T_{\sigma(r)}(b^*), c^*\rangle = \sum_{i, j}\langle b_i \otimes b_j \otimes a_i \cdot_{A} a_j, a^* \otimes b^* \otimes c^* \rangle, \\
		&\langle T_{\sigma(r)}( L^*(T_{r}(a^*)) b^*), c^*\rangle = \sum_{i,j}\langle a_i \otimes b_i \cdot_{A} b_j \otimes a_j, a^* \otimes b^* \otimes c^* \rangle, \\
		&\langle T_{\sigma(r)}( \ad^*(T_{\sigma(r)}(b^*)) a^* ), c^*\rangle = \sum_{i, j}\langle (a_i \cdot_{A} b_j - b_j \cdot_{A} a_i) \otimes b_i \otimes a_j, a^* \otimes b^* \otimes c^*\rangle.
	\end{align*}
	Then
	\begin{align*}
		&\langle -T_{\sigma(r)}(a^*) \cdot_A T_{\sigma(r)}(b^*) + T_{\sigma(r)}(L^*(T_{r}(a^*)) b^* + \ad^*(T_{\sigma(r)}(b^*)) a^*), c^*\rangle \\
		= & \sum_{i,j}\langle -b_i \otimes a_i \cdot_{A} a_j \otimes b_j + a_i \otimes a_j \otimes b_i \cdot_{A} b_j + (a_i \cdot_{A} b_j - b_j \cdot_{A} a_i) \otimes a_j \otimes b_i, a^* \otimes c^* \otimes b^*\rangle \\
		= & \langle \sum_j( (\id \otimes R(a_j) - \ad(a_j) \otimes \id)(r - \sigma(r) ) ) \otimes b_j -[[r, r]], a^* \otimes c^* \otimes b^*\rangle.
	\end{align*}
	Therefore, if $r-\sigma(r)$ is $(R, \ad)$-invariant, then $r$ is a solution of the perm Yang-Baxter equation if and only if Eq.~\eqref{eq:pybeofs} holds.
\end{proof}

\begin{lemma}\label{lem:pcbd2c}
	Let $(A, \cdot_{A})$ be a perm algebra and $r \in A \otimes A$.
	Define a comultiplication $\Delta_r: A \to A \otimes A$ by Eq.~\eqref{eq:pcbd}.
	Then the multiplication $\cdot_r: A^* \otimes A^* \to A$ on $A^*$ obtained as the linear dual of $\Delta_r$ is explicitly given by
	\begin{equation}
		a^* \cdot_r b^* = L^*(T_{r}(a^*)) b^* + \ad^*(T_{\sigma(r)}(b^*)) a^*, \;\; \forall a^*, b^* \in A^*. \label{eq:pcbdr}
	\end{equation}
\end{lemma}
\begin{proof}
	Set $r=\sum_{i} a_i \otimes b_i$. 
	For all $a \in A$ and $a^*, b^* \in A^*$, we have
	\begin{align*}
		&\langle a, L^*(T_{r}(a^*)) b^* + \ad^*(T_{\sigma(r)}(b^*)) a^* \rangle \\
		&= \sum_{i} \langle a^*, a_i\rangle\langle a, L^*(b_i)b^*\rangle + \sum_{i} \langle b^*, b_i\rangle\langle a, \ad^*(a_i)a^*\rangle  \\
		&=\sum_{i} \langle a_i \otimes b_i \cdot_{A} a + (a_i \cdot_{A} a - a \cdot_A a_i)\otimes b_i, a^* \otimes b^*\rangle  =\langle \Delta_r(a), a^* \otimes b^*\rangle = \langle a, a^* \cdot_r b^*\rangle.
	\end{align*}
	Therefore, Eq.~\eqref{eq:pcbdr} holds.
\end{proof}

\begin{proposition}\label{prop:qtpb2h}
	Let $(A, \cdot_A, \Delta_r)$ be a quasi-triangular perm bialgebra.
	Then $T_{r},T_{\sigma(r)}: A^* \to A$ are both perm algebra homomorphism between $(A^*, \cdot_r)$ and $(A, \cdot_{A})$, where the multiplication $\cdot_r$ is the linear dual of $\Delta_r$.
\end{proposition}
\begin{proof}
	By hypothesis, $r$ is a solution of the perm Yang-Baxter equation in $(A, \cdot_{A})$ and its skew-symmetric part is $(R, \ad)$-invariant.
	Combining Proposition~\ref{prop:pybeof} and Lemma~\ref{lem:pcbd2c}, we obtain 
	\begin{equation*}
		T_{r}(a^*) \cdot_A T_{r}(b^*) = T_{r}(a^* \cdot_{r} b^*), \;
		T_{\sigma(r)}(a^*) \cdot_A T_{\sigma(r)}(b^*) = T_{\sigma(r)}(a^* \cdot_{r} b^*), \;\;
		\forall a^*, b^* \in A^*,
	\end{equation*}
	which implies both $T_{r}$ and $T_{\sigma(r)}$ are perm algebra homomorphisms.
\end{proof}

We conclude this section by establishing a characterization of the $(R, \ad)$-invariant condition.
\begin{lemma}\label{lem:radi}
	Let $(A, \cdot_A)$ be a perm algebra and $r \in A \otimes A$.
	Then $r$ is $(R, \ad)$-invariant if and only if 
	\begin{equation}
		T_{r}(\ad^*(a) a^*) -  T_{r}(a^*) \cdot_A a = 0, \;\; \forall a \in A, a^* \in A^*. \label{eq:adi} 
	\end{equation}
	If in addition $r$ is skew-symmetric, then the following conditions are equivalent.
	\begin{enumerate}
		\item 
		\label{it:ad1}
		$r$ is $(R, \ad)$-invariant.
		
		\item 
		\label{it:ad2}
		The following equation holds for all $a \in A$ and $a^* \in A^*$:
		\begin{equation}
			T_{r} (R^*(a)a^*) - \ad(a) T_{r}(a^*) = 0, \;\; \forall a \in A, \label{eq:sadi}
		\end{equation}
		
		\item 
		\label{it:ad3}
		The following equation holds for all $a \in A$ and $a^* \in A^*$:
		\begin{equation}
			L^*(T_{r}(a^*))b^* - \ad^*(T_{r}(b^*))a^* = 0. \label{eq:sadii}
		\end{equation}
	\end{enumerate}
\end{lemma}
\begin{proof}
	For all $a \in A$ and $a^*, b^* \in A^*$, we have 
	\begin{align*}
		\langle T_{r}(\ad^*(a) a^*) -  T_{r}(a^*) \cdot_A a, b^*\rangle &= \langle r, \ad^*(a) a^* \otimes b^* - a^* \otimes R^*(a) b^* \rangle \\
		&= \langle (\ad(a) \otimes \id - \id \otimes R(a))r, a^* \otimes b^* \rangle,
	\end{align*}
	which shows that $r$ is $(R, \ad)$-invariant if and only if Eq.~\eqref{eq:adi} holds.
	
	Suppose that $r$ is skew-symmetric. 
	For all $a \in A$ and $a^*, b^* \in A^*$, we have
	\begin{align*}
		&\langle T_{r}(R^*(a)a^*) - \ad(a) T_{r}(a^*), b^*\rangle = \langle r, R^*(a)a^* \otimes b^* - a^* \otimes \ad^*(a)b^*\rangle \\
		&= \langle r, -b^* \otimes R^*(a)a^*  + \ad^*(a)b^* \otimes a^* \rangle = -\langle (\id \otimes R(a) - \ad(a) \otimes \id)(r), b^* \otimes a^*\rangle,
	\end{align*}
	which implies the equivalence between (\ref{it:ad1}) and (\ref{it:ad2}).
	Similarly, we show the equivalence between (\ref{it:ad1}) and (\ref{it:ad3}).
	The proof is completed.
\end{proof}
\begin{remark}\label{rmk:radi}
	Let $(A, \cdot_A)$ be a perm algebra, and $r \in A \otimes A$ be skew-symmetric and $(R, \ad)$-invariant.
	Combining Eqs.~\eqref{eq:adi} and \eqref{eq:sadi} in Lemma~\ref{lem:radi}, we have
	\begin{equation}
		T_{r}(L^*(a)a^*) - a \cdot_{A} T_{r}(a^*)= 0, \;\; \forall a \in A, \; a^* \in A^*. \label{eq:adil}
	\end{equation}
\end{remark}

%%%%%%%%%%%%%%%%%%%%%%%%%%%%%%%%%%%%%%%%%%%%%%%%%%%%%%%%%%%%%%%%%%%%%%%%%%%%%%%%
\section{Factorizable perm bialgebras}\label{sec:fpb}
In this section, we introduce the notion of factorizable perm bialgebras as a subclass of quasi-triangular perm bialgebras, which leads to a factorization of the underlying perm algebra.
We show that the double of a perm bialgebra admits a canonical factorizable perm bialgebra structure.

\begin{definition}
	A quasi-triangular perm bialgebra $(A, \cdot_A, \Delta_r)$ is called \textbf{factorizable} if $T_{r-\sigma(r)}: A^* \to A$ is a linear isomorphism of vector spaces.
\end{definition}

\begin{remark}
	Let $(A, \cdot_A, \Delta_r)$ be a quasi-triangular perm bialgebra.
	Note that $T_{r - \sigma(r)} = T_r - T_{\sigma(r)}$.
	The key distinction between triangular and factorizable perm bialgebra structures lies in the properties of $T_{r - \sigma(r)}$.
	Explicitly, $T_{r - \sigma(r)} = 0$ implies symmetry of $r$ and thus a triangular perm bialgebra, while the nondegeneracy of $T_{r - \sigma(r)}$ defines a factorizable perm bialgebra.
\end{remark}

As a subclass of quasi-triangular perm bialgebras, factorizable perm bialgebras appear in pairs similarly.
\begin{corollary}\label{coro:rsrf}
	If $(A, \cdot_A, \Delta_r)$ is a factorizable perm bialgebra, then $(A, \cdot_A, \Delta_{\sigma(r)})$ is also a factorizable perm bialgebra
\end{corollary}
\begin{proof}
	If $T_{r-\sigma(r)}$ is a linear isomorphism of vector spaces, then 
	\begin{equation*}
		T_{\sigma(r) - \sigma(\sigma(r))} = - T_{r - \sigma(r)}
	\end{equation*}
	is also a linear isomorphism of vector spaces.
	Therefore, the conclusion follows from Proposition~\ref{prop:rsrqt}.
\end{proof}

Let $(A, \cdot_{A})$ be a perm algebra and $r \in A \otimes A$.
Define a linear map $T_{r} \oplus T_{\sigma(r)}: A ^* \to A \oplus A$ by 
\begin{equation*}
	T_{r} \oplus T_{\sigma(r)}(a^*) = (T_{r}(a^*), T_{\sigma(r)}(a^*)), \;\; \forall a^* \in A^*.
\end{equation*}

The following result justifies the terminology of a factorizable perm bialgebra.
\begin{proposition}
	Let $(A, \cdot_A, \Delta_r)$ be a factorizable perm bialgebra.
	Then $\mathrm{Im}(T_{r} \oplus T_{\sigma(r)})$ is a perm subalgebra of the direct sum perm algebra $A \oplus A$, and is isomorphic to the perm algebra $(A^*, \cdot_r)$, where $\cdot_r: A^* \otimes A^* \to A^*$ is defined by Eq.~\eqref{eq:pcbdr}.
	Moreover, any $a \in A$ has a unique decomposition 
	$a = a_{1} - a_{2}$ with $(a_{1}, a_{2}) \in \mathrm{Im}(T_{r} \oplus T_{\sigma(r)})$.
\end{proposition}
\begin{proof}
	By Proposition~\ref{prop:qtpb2h}, $T_{r} \oplus T_{\sigma(r)}$ is a perm algebra homomorphism from $(A^*, \cdot_r)$ to the direct sum perm algebra $A \oplus A$.
	Since $T_{r - \sigma(r)}: A^* \to A$ is nondegenerate, $\ker(T_{r} \oplus T_{\sigma(r)}) = 0$.
	Thus, $\mathrm{Im}(T_{r} \oplus T_{\sigma(r)})$ is isomorphic to the perm algebra $(A^*, \cdot_r)$ as perm algebras.
	Moreover, for any $a \in A$, we have
	\begin{equation*}
		a = T_{r-\sigma(r)} T^{-1}_{r-\sigma(r)}(a) = T_{r}T^{-1}_{r-\sigma(r)}(a) - T_{\sigma(r)} T^{-1}_{r-\sigma(r)}(a) = a_1 - a_2,
	\end{equation*}
	where $(a_1, a_2) =(T_{r}T^{-1}_{r-\sigma(r)}(a), T_{\sigma(r)}T^{-1}_{r-\sigma(r)}(a))  \in \mathrm{Im}( T_{r} \oplus T_{\sigma(r)} )$.
	The uniqueness follows from the nondegeneracy of $T_{r-\sigma(r)}$ again.
\end{proof}

Let $((A, \cdot_A), (A^*, \cdot_{A^*}))$ be a perm bialgebra.
Define a multiplication on $\mathfrak{A} = A \oplus A^*$ by
\begin{equation*}
	(a+a^*) \cdot_\mathfrak{A} (b+b^*) = (a \cdot_A b + L_{\cdot_{A^*}}^*(a^*)b + \ad_{\cdot_{A^*}}^*(b^*)a) + (a^* \cdot_{A^*} b^* + L_{\cdot_{A}}^*(a)b^* + \ad_{\cdot_{A}}^*(b)a^*), 
\end{equation*}
where $a, b \in A, a^*, b^* \in A^*$.
Then $(\mathfrak{A}, \cdot_\mathfrak{A})$ is a perm algebra \cite{lin2025infinite}, called the \textbf{double perm algebra} of the perm bialgebra $((A, \cdot_A), (A^*, \cdot_{A^*}))$.

\begin{proposition}
	Let $(A, \cdot_{A}, \Delta_r)$ be a factorizable perm bialgebra.
	Then the double perm algebra $(\mathfrak{A}, \cdot_{\mathfrak{A}})$ is isomorphic to the direct sum perm algebra $A \oplus A$.
\end{proposition}
\begin{proof}
	Define a linear map $\varphi: A \oplus A^* \to A \oplus A$ by 
	\begin{equation*}
		\varphi(a) = (a, a), \; \varphi(a^*) = (T_{r}(a^*), T_{\sigma(r)}(a^*)), \;\; \forall a \in A, a^* \in A^*.
	\end{equation*}
	By the fact that $T_{r - \sigma(r)}$ is a linear isomorphism of vector spaces, we have $\ker(\varphi) = 0$ and $\varphi$ is a linear isomorphism of vector spaces.
	It is easy to show that both $\varphi|_A$ and $\varphi|_{A^*} = T_{r} \oplus T_{\sigma(r)}$ are homomorphisms of perm algebras by Proposition~\ref{prop:qtpb2h}. 
	Moreover, for all $a, b \in A$ and $a^*, b^* \in A^*$, we have
	\begin{align*}
		&\langle \ad_{\cdot_{r}}^*(b^*)a + T_{r}(L_{\cdot_{A}}^*(a)b^*), a^*\rangle = \langle a, b^* \cdot_{r} a^* - a^* \cdot_{r} b^*\rangle + \langle r,  L_{\cdot_{A}}^*(a)b^* \otimes a^* \rangle \\
		&= \langle \Delta_r(a), b^* \otimes a^* - a^* \otimes b^*\rangle + \langle (L_{\cdot_{A}}(a) \otimes \id)(r),  b^* \otimes a^* \rangle \\
		&= \langle ( R_{\cdot_{A}}(a) \otimes \id -  \id \otimes \ad_{\cdot_{A}}(a) + \id \otimes L_{\cdot_{A}}(a))(\sigma(r)) - ( \id \otimes R_{\cdot_{A}}(a) - \ad_{\cdot_{A}}(a) \otimes \id )(r) , a^* \otimes b^* \rangle \\
		&\overset{\eqref{eq:rsrradi}}{=} \langle ( L_{\cdot_{A}}(a) \otimes \id )(\sigma(r)), a^* \otimes b^*\rangle = \langle r, b^* \otimes L_{\cdot_{A}}^*(a)a^*\rangle = \langle a \cdot_{A} T_{r}(b^*), a^*\rangle,
	\end{align*}
	that is, 
	\begin{equation*}
		\ad_{\cdot_{r}}^*(b^*)a + T_{r}(L_{\cdot_{A}}^*(a)b^*) = a \cdot_{A} T_{r}(b^*).
	\end{equation*}
	Similarly, we have 
	\begin{equation*}
		\ad_{\cdot_{r}}^*(b^*)a + T_{\sigma(r)}(L_{\cdot_{A}}^*(a)b^*) = a \cdot_{A} T_{\sigma(r)}(b^*).
	\end{equation*}
	Thus, we have
	\begin{align*}
		&\varphi(a \cdot_{\mathfrak{A}} b^*) = \varphi(\ad_{\cdot_{r}}^*(b^*)a + L_{\cdot_{A}}^*(a)b^*) =(\ad_{\cdot_{r}}^*(b^*)a + T_{r}(L_{\cdot_{A}}^*(a)b^*), \ad_{\cdot_{r}}^*(b^*)a + T_{\sigma(r)}(L_{\cdot_{A}}^*(a)b^*)) \\
		&=(a \cdot_{A} T_{r}(b^*), a \cdot_{A} T_{\sigma(r)}(b^*)) = (a, a) \cdot (T_{r}(b^*), T_{\sigma(r)}(b^*)) = \varphi(a) \cdot \varphi(b^*).
	\end{align*}
	By a similar argument, we show that 
	\begin{equation*}
		\varphi(b^* \cdot_{\mathfrak{A}} a) = \varphi(b^*) \cdot \varphi(a).
	\end{equation*} 
	In conclusion, $\varphi: \mathfrak{A} \to A \oplus A$ is an isomorphism of perm algebras.
	The proof is completed.
\end{proof}

The following theorem establishes the existence of a canonical factorizable perm bialgebra structure on the double perm algebra of a perm bialgebra.
\begin{theorem}
	Let $((A, \cdot_A), (A^*, \cdot_{A^*}))$ be a perm bialgebra.
	Let $\{e_1, e_2, \cdots, e_n\}$ be a basis of $A$ and $\{e_1^*, e_2^*, \cdots, e_n^*\}$ be the dual basis of $A^*$.
	Set 
	\begin{equation*}
		r = \sum_{i}e_i \otimes e_i^* \in A \otimes A^* \subset \mathfrak{A} \otimes \mathfrak{A}.
	\end{equation*}
	Then $(\mathfrak{A}, \cdot_\mathfrak{A}, \Delta_r)$ is a factorizable perm bialgebra, where $\Delta_r$ is defined by Eq.~\eqref{eq:pcbd} via $r$.
\end{theorem}
\begin{proof}
	For all $a \in A$ and $a^* \in A^*$, noting that 
	\begin{align*}
		&\sum_{i} e_i \otimes L^*_{\cdot_{A^*}}(e_i^*)a = \sum_{i} R^*_{\cdot_{A^*}}(e_i^*)a \otimes e_i, \;\; \sum_{i} e_i \otimes \ad^*_{\cdot_{A}}(a)e_i^* = \sum_{i} \ad_{\cdot_{A}}(a)e_i \otimes e_i^*, \\
		&\sum_{i} R^*_{\cdot_{A}}(a)e_i^* \otimes e_i = \sum_{i} e_i^* \otimes R_{\cdot_{A}}(a)e_i, 
	\end{align*}
	we have
	\begin{align*}
		&(\id \otimes R_{\cdot_{\mathfrak{A}}} (a) - \ad_{\cdot_{\mathfrak{A}}}(a) \otimes \id)(r- \sigma(r)) \\
		&=\sum_{i}( e_i \otimes e^*_i \cdot_{\mathfrak{A}} a - e_i^* \otimes e_i \cdot_{\mathfrak{A}} a - (a \cdot_{\mathfrak{A}} e_i - e_i \cdot_{\mathfrak{A}} a) \otimes e_i^* + (a \cdot_{\mathfrak{A}} e_i^* - e_i^* \cdot_{\mathfrak{A}} a) \otimes e_i )\\
		&=\sum_{i}( e_i \otimes ( L^*_{\cdot_{A^*}}(e_i^*) a + \ad_{\cdot_{A}}^*(a)e_i^* ) - e_i^* \otimes e_i \cdot_{A} a - \ad(a)e_i \otimes e_i^* \\
		&\quad + (L_{\cdot_{A}}^*(a)e_i^*+\ad^*_{\cdot_{A^*}}(e_i^*)a - L^*_{\cdot_{A^*}}(e_i^*)a - \ad_{\cdot_{A}}^*(a)e_i^*) \otimes e_i )\\
		&=\sum_{i} ( e_i \otimes ( L^*_{\cdot_{A^*}}(e_i^*) a + \ad_{\cdot_{A}}^*(a)e_i^* ) - e_i^* \otimes e_i \cdot_{A} a - \ad(a)e_i \otimes e_i^* + (R_{\cdot_{A}}^*(a)e_i^* - R^*_{\cdot_{A^*}}(e_i^*)a ) \otimes e_i ) \\
		&= 0.
	\end{align*}
	Similarly, we show that 
	\begin{equation*}
		(\id \otimes R_{\cdot_{\mathfrak{A}}} (a^*) - \ad_{\cdot_{\mathfrak{A}}}(a^*) \otimes \id)(r- \sigma(r))  = 0.
	\end{equation*}
	Hence, the skew-symmetric part of $r$ is $(R_{\cdot_\mathfrak{A}}, \ad_{\cdot_\mathfrak{A}})$-invariant.
	
	Next, we prove that $r$ satisfies the perm Yang-Baxter equation in $(\mathfrak{A}, \cdot_\mathfrak{A})$. In fact, 
	\begin{align*}
		&[[r, r]] = \sum_{i, j} ( e_i \otimes e_i^* \cdot_{\mathfrak{A}} e_j \otimes e_j^* - e_i \otimes e_j \otimes e_i^* \cdot_{\mathfrak{A}} e_j^* + e_i \cdot_{\mathfrak{A}} e_j \otimes e_i^* \otimes e_j^* -  e_j \cdot_{\mathfrak{A}} e_i \otimes e_i^* \otimes e_j^* ) \\
		&= \sum_{i, j}( e_i \otimes (L^*_{\cdot_{A^*}}(e_i^*)e_j + \ad^*_{\cdot_{A}}(e_j) e_i^*) \otimes e_j^* - e_i \otimes e_j \otimes e_i^* \cdot_{A^*} e_j^* + e_i \cdot_{A} e_j \otimes e_i^* \otimes e_j^* -  e_j \cdot_{A} e_i \otimes e_i^* \otimes e_j^* ) \\
		&= \sum_{i, j}( e_i \otimes (L^*_{\cdot_{A^*}}(e_i^*)e_j + \ad^*_{\cdot_{A}}(e_j) e_i^*) \otimes e_j^* - e_i \otimes e_j \otimes L_{\cdot_{A^*}}(e_i^*) e_j^* - \ad_{\cdot_{A}}(e_j)e_i  \otimes e_i^* \otimes e_j^* ) \\
		&= 0,
	\end{align*}
	that is, $r$ is a solution of the perm Yang-Baxter equation in $(\mathfrak{A}, \cdot_{\mathfrak{A}})$.
	Therefore, $(\mathfrak{A}, \cdot_{\mathfrak{A}}, \Delta_r)$ is a quasi-triangular perm bialgebra.

	Finally, we have to show that $T_{r - \sigma(r)}$ is a linear isomorphism of vector spaces.
	Note that $T_{r}, T_{\sigma(r)}: \mathfrak{A}^* \to \mathfrak{A}$ are given by
	\begin{equation*}
		T_{r}(a^*+a) = a^*, \;
		T_{\sigma(r)}(a^*+a) = a, \;\;
		\forall a \in A, a^* \in A^*.
	\end{equation*}
	Thus, we have $T_{r - \sigma(r)}(a^* + a) = a^* - a$, which implies that the linear map $T_{r - \sigma(r)}: \mathfrak{A}^* \to \mathfrak{A}$ is a linear isomorphism of vector spaces.
	Therefore, the perm bialgebra $(\mathfrak{A}, \cdot_\mathfrak{A}, \Delta_r)$ is factorizable.
\end{proof}
%%%%%%%%%%%%%%%%%%%%%%%%%%%%%%%%%%%%%%%%%%%%%%%%%%%%%%%%%%%%%%%%%%%%%%%%%%%%%%%%
%%%%%%%%%%%%%%%%%%%%%%%%%%%%%%%%%%%%%%%%%%%%%%%%%%%%%%%%%%%%%%%%%%%%%%%%%%%%%%%%
%%%%%%%%%%%%%%%%%%%%%%%%%%%%%%%%%%%%%%%%%%%%%%%%%%%%%%%%%%%%%%%%%%%%%%%%%%%%%%%%
\section{Relative Rota-Baxter operators and the perm Yang-Baxter equation}\label{sec:rboqrpa}
In this section, we further study operator forms of solutions of the perm Yang-Baxter equation in perm algebras under the $(R, \ad)$-invariant condition.
We show that relative Rota-Baxter operators of weights on perm algebras could be used to characterize solutions of the perm Yang-Baxter equation whose skew-symmetric parts are $(R, \ad)$-invariant.
Furthermore, for quadratic perm algebras, we demonstrate that these operators are in one-to-one correspondence with linear transformations fulfilling a Rota-Baxter-type identity (see Eq.~\eqref{eq:prbw}).

%%%%%%%%%%%%%%%%%%%%%%%%%%%%%%%%%%%%%%%%%%%%%%%%%%%%%%%%%%%%%%%%%%%%%%%%%%%%%%%%
\begin{definition}
	A \textbf{representation} of a perm algebra $(A, \cdot_A)$ is a triple $(V, l, r)$, where $V$ is a vector space and $l, r: A \to \End(V)$ are linear maps satisfying
	\begin{eqnarray}
		&& l(a \cdot_A b) = l(a) l(b) = l(b) l(a), \label{eq:prep1}\\
		&& r(a \cdot_{A} b) = r(b) r(a) = r(b) l(a) = l(a) r(b), \;\; \forall a, b \in A. \label{eq:prep2}
	\end{eqnarray}
\end{definition}

Let $(A, \cdot_A)$ be a perm algebra, $V$ be a vector space and $l, r: A \to \End(V)$ be linear maps.
Define a multiplication $\cdot: (A \oplus V) \otimes (A \oplus V) \to A \oplus V$ on $A \oplus V$ by
\begin{equation*}
	(a_1 + v_1) \cdot (a_2 + v_2) := a_1 \cdot_{A} a_2 + (l(a_1) v_2 + r(a_2) v_1), \;\; \forall a_1, a_2 \in A, \; v_1, v_2 \in V.
\end{equation*}
Then $(V, l, r)$ is a representation of $(A, \cdot_A)$ if and only if $(A \oplus V, \cdot)$ is a perm algebra.
In such a case, the resulting perm algebra $(A \oplus V, \cdot)$ is called the \textbf{semi-direct product perm algebra by $(A, \cdot_{A})$ and $(V, l, r)$} and denoted by $(A \ltimes_{l, r} V, \cdot)$.

\begin{example}
	Let $(A, \cdot_{A})$ be a perm algebra.
	Then $(A, L, R)$ is a representation of $(A, \cdot_{A})$.
\end{example}

\begin{proposition}\label{prop:cadjr}
	{\rm (\cite{lin2025infinite})}
	Let $(A, \cdot_A)$ be a perm algebra and $(V, l, r)$ be a representation of $(A, \cdot_A)$.
	Then $(V^*, l^*, l^* - r^*)$ is a representation of $(A, \cdot_A)$.
	In particular, $(A^*, L^*, \ad^*)$ is a representation of $(A, \cdot_A)$.
\end{proposition}

Inspired by the notion of $\mathfrak{g}$-Lie algebras \cite{bai2010nonabelian, knapp2002lie}, we introduce the notion of $A$-perm algebras.
\begin{definition}
	Let $(A, \cdot_{A})$ and $(V, \cdot_{V})$ be perm algebras, and $l, r: A \to \End(V)$ be linear maps.
	If $(V,l,r)$ is a representation of $(A, \cdot_{A})$ and the following equations hold:
	\begin{align}
		&r(a)(u \cdot_{V} w) = r(a)(w \cdot_{V} u) = u \cdot_{V} (r(a)w), \label{eq:apar} \\
		&(l(a)u) \cdot_{V} w = (r(a)u) \cdot_{V} w = l(a)(u \cdot_{V} w) = u \cdot_{V} (l(a)w), \;\; \forall a \in A, \; u, w \in V, \label{eq:apal} 
	\end{align}
	then the quadruple $(V,l,r,\cdot_{V})$ is called an \textbf{$A$-perm algebra}.
\end{definition}

\begin{example}
	Let $(A, \cdot_{A})$ be a perm algebra.
	Then $(A, L, R, \cdot_{A})$ is an $A$-perm algebra.
\end{example}

Similar to representations of perm algebras, an $A$-perm algebra admits an equivalent characterization in terms of the perm algebra structure on the direct sum $A \oplus V$ of vector spaces, as stated in the following proposition.
\begin{proposition}
	Let $(A, \cdot_{A})$ and $(V, \cdot_{V})$ be perm algebras, and $l, r: A \to \End(V)$ be linear maps.
	Then $(V,l,r,\cdot_{V})$ is an $A$-perm algebra if and only if $(A \oplus V, \cdot)$ is a perm algebra, where the multiplication $\cdot$ on $A \oplus V$ is defined by 
	\begin{equation*}
		(a+u) \cdot (b+v) := [a, b] + l(a)v + r(b)u + [u, v], \;\; \forall a, b \in A, \; u, v \in V.
	\end{equation*}
\end{proposition}
\begin{proof}
	It follows from a straightforward computation.
\end{proof}

\begin{proposition}\label{prop:r2apa}
	Let $(A, \cdot_{A})$ be a perm algebra, and $r \in A \otimes A$ be skew-symmetric and $(R, \ad)$-invariant.
	Define a multiplication $\diamond_r: A^* \otimes A^* \to A^*$ on $A^*$ by
	\begin{equation*}
		a^* \diamond_r b^* = L^*(T_r(a^*))b^*, \;\; \forall a^*, b^* \in A^*.
	\end{equation*}
	Then $(A^*, L^*, \ad^*, \diamond_r)$ is an $A$-perm algebra.
\end{proposition}
\begin{proof}
	By Remark~\ref{rmk:radi}, Eq.~\eqref{eq:adil} holds.
	Thus, we have
	\begin{align*}
		(a^* \diamond_r b^*) \diamond_r c^* &= L^*(T_{r}( L^*(T_{r}(a^*)) b^*  )) c^* \overset{\eqref{eq:adil}}{=} L^*( T_{r}(a^*) \cdot_{A} T_{r}(b^*) ) c^*, \\
		(b^* \diamond_r a^*) \diamond_r c^* &= L^*(T_{r}( L^*(T_{r}(b^*)) a^*  )) c^* \overset{\eqref{eq:adil}}{=} L^*( T_{r}(b^*) \cdot_{A} T_{r}(a^*) ) c^*, \\
		a^* \diamond_r (b^* \diamond_r c^*) &= L^*( T_{r}(a^*)) L^*(  T_{r}(b^*) ) c^*.
	\end{align*}
	Note that $(A^*, L^*, \ad^*)$ is a representation of $(A, \cdot_{A})$.
	Therefore, we have
	\begin{equation*}
		(a^* \diamond_r b^*) \diamond_r c^* = (b^* \diamond_r a^*) \diamond_r c^* = a^* \diamond_r (b^* \diamond_r c^*).
	\end{equation*}
	That is, $(A^*, \diamond_r)$ is a perm algebra.
	Moreover, we have 
	\begin{align*}
		&\langle \ad^*(a)(a^* \diamond_r b^*), b \rangle = \langle L^*(T_r(a^*))b^*, a \cdot_{A} b - b \cdot_{A} a \rangle = \langle b^*, T_r(a^*) \cdot_{A} (a \cdot_{A} b - b \cdot_{A} a) \rangle,\\
		&\langle \ad^*(a)(b^* \diamond_r a^*), b\rangle = \langle L^*(T_r(b^*))a^*, a \cdot_{A} b - b \cdot_{A} a \rangle  \overset{\eqref{eq:sadii}}{=} \langle \ad^*(T_r(a^*))b^*, a \cdot_{A} b - b \cdot_{A} a \rangle \\
		&=\langle b^*, T_r(a^*) \cdot_{A} (a \cdot_{A} b - b \cdot_{A} a) \rangle, \\
		&\langle a^* \diamond_r (\ad^*(a)b^*), b\rangle = \langle L^*(T_r(a^*)) (\ad^*(a)b^*), b \rangle = \langle \ad^*(a)b^*, T_r(a^*) \cdot_{A} b\rangle \\
		&= \langle b^*, \ad(a)(T_{r}(a^*) \cdot_{A} b) \rangle = \langle b^*, a \cdot_{A} (T_{r}(a^*) \cdot_{A} b) - (T_{r}(a^*) \cdot_{A} b) \cdot_{A} a\rangle.
	\end{align*}
	Thus, Eq.~\eqref{eq:apar} holds for $(A^*, L^*, \ad^*, \diamond_r)$.
	Similarly, we show that Eq.~\eqref{eq:apal} holds for $(A^*, L^*, \ad^*, \diamond_r)$.
	Therefore, $(A^*, L^*, \ad^*, \diamond_r)$ is an $A$-perm algebra.
\end{proof}

\begin{definition}
	Let $(A, \cdot_{A})$ be a perm algebra and $(V, l, r, \cdot_{V})$ be an $A$-perm algebra.
	A linear map $T: V \to A$ is called a \textbf{relative Rota-Baxter operator of weight $\lambda$} on $(A, \cdot_{A})$ with respect to $(V, l, r, \cdot_V)$ if $T$ satisfies
	\begin{equation*}
		T(u) \cdot_{A} T(w) = T( l(T(u))w + r(T(w))u + \lambda u \cdot_{V} w ), \;\; \forall u, w \in V.
	\end{equation*}
	In particular, a relative Rota-Baxter operator $P: A \to A$ of weight $\lambda$ on $(A, \cdot_{A})$ with respect to $(A, L, R, \cdot_{A})$ is called a \textbf{Rota-Baxter operator of weight $\lambda$} on $(A, \cdot_{A})$, namely, a linear map $P: A \to A$ is called a Rota-Baxter operator of weight $\lambda$ on $(A, \cdot_{A})$ if
	\begin{equation*}
		P(a) \cdot_A P(b) = P(P(a) \cdot_A b + a \cdot_A P(b) + \lambda a \cdot_A b), 
		\;\; \forall a, b \in A.
	\end{equation*}
\end{definition}

\begin{remark}
	If the multiplication $\cdot_{V}$ on $V$ is trivial, that is, $u \cdot_{V} w = 0$ for all $u, w \in V$, then $T$ is simply called an \textbf{$\mathcal{O}$-operator of $(A, \cdot_{A})$ associated to $(V, l, r)$} \cite{lin2025infinite}.
\end{remark}

\begin{theorem}\label{thm:pybe2rrb}
	Let $(A, \cdot_{A})$ be a perm algebra and $r \in A \otimes A$ such that its skew-symmetric part is $(R, \ad)$-invariant.
	Then the following conditions are equivalent.
	\begin{enumerate}
		\item 
		$r$ is a solution of the perm Yang-Baxter equation in $(A, \cdot_{A})$ such that $(A, \cdot_{A}, \Delta_r)$ is a quasi-triangular perm bialgebra.
		
		\item 
		$T_{r}$ is a relative Rota-Baxter operator of weight $-1$ on $(A, \cdot_{A})$ with respect to the $A$-perm algebra $(A^*, L^*, \ad^*, \diamond_{r - \sigma(r)})$ where the multiplication $\diamond_{r - \sigma(r)}$ is defined by 
		\begin{equation}
			a^* \diamond_{r - \sigma(r)} b^* = L^*(T_{r-\sigma(r)}(a^*))b^*, \;\; \forall a^*, b^* \in A^*. \label{eq:apa2mul}
		\end{equation}
		That is, the following equation holds:
		\begin{equation}
			T_{r}(a^*) \cdot_{A} T_{r}(b^*) = T_{r}( L^*(T_{r}(a^*)b^* + \ad^*(T_{r}(b^*))a^* - a^* \diamond_{r-\sigma(r)} b^* ), \;\; \forall a^*, b^* \in A^*. \label{eq:pybe2rrb}
		\end{equation}
	\end{enumerate}
\end{theorem}
\begin{proof}
	Note that $r-\sigma(r)$ is skew-symmetric.
	By Proposition~\ref{prop:r2apa}, $(A^*, L^*, \ad^*, \diamond_{r-\sigma(r)})$ is an $A$-perm algebra.
	For all $a^*, b^* \in A^*$, we have 
	\begin{align*}
		&\;\; T_{r}( L^*(T_{r}(a^*)b^* + \ad^*(T_{r}(b^*))a^* - a^* \diamond_{r-\sigma(r)} b^* ) \\
		&\overset{\hphantom{\eqref{eq:sadii}}}{=} T_{r}( L^*(T_{r}(a^*)b^* + \ad^*(T_{r}(b^*))a^* - L^*(T_{r-\sigma(r)}(a^*))b^* ) \\
		&\overset{\eqref{eq:sadii}}{=} T_{r}( L^*(T_{r}(a^*)b^* + \ad^*(T_{r}(b^*))a^* - \ad^*(T_{r-\sigma(r)}(b^*))a^* ) \\
		&\overset{\hphantom{\eqref{eq:sadii}}}{=} T_{r}( L^*(T_{r}(a^*)b^* + \ad^*(T_{\sigma(r)}(b^*))a^* ).
	\end{align*}
	Thus, Eq.~\eqref{eq:pybe2rrb} holds if and only if Eq.~\eqref{eq:pybeof}.
	Therefore, the theorem follows immediately from Proposition~\ref{prop:pybeof}.
\end{proof}

When taking a symmetric 2-tensor $r \in A \otimes A$ in Theorem~\ref{thm:pybe2rrb}, we recover the following result established in \cite{lin2025infinite}.
\begin{corollary}
	{\rm (\cite[Corollary~3.35]{lin2025infinite})}
	Let $(A, \cdot_{A})$ be a perm algebra and $r \in A \otimes A$ be symmetric.
	Then $r$ is a solution of the perm Yang-Baxter equation in $(A, \cdot_{A})$ if and only if $T_{r}: A^* \to A$ is an $\mathcal{O}$-operator of $(A, \cdot_{A})$ associated to $(A^*, L^*, \ad^*)$.
\end{corollary}

Recall that a bilinear form $\mathfrak{B} \in \otimes^2 A^*$ on a perm algebra $(A, \cdot_A)$ is called \textbf{invariant} if 
\begin{equation*}
	\mathfrak{B}(a \cdot_{A} b, c) = \mathfrak{B}(a, b \cdot_{A} c - c \cdot_{A} b), 
	\;\; \forall a, b, c \in A.
\end{equation*}
A \textbf{quadratic perm algebra} $(A, \cdot_A, \mathfrak{B})$ is a perm algebra $(A, \cdot_A)$ equipped with a nondegenerate skew-symmetric invariant bilinear form $\mathfrak{B} \in \otimes^2 A^*$.

Let $\mathfrak{B}$ be a nondegenerate bilinear form on a vector space $A$.
Then there is an isomorphism $\mathfrak{B}^\sharp: A \to A^*$ given by 
\begin{equation*}
	\mathfrak{B}(a, b) = \langle \mathfrak{B}^\sharp(a), b\rangle, \;\; \forall a, b \in A.
\end{equation*}
Define a 2-tensor $\phi_\mathfrak{B} \in A \otimes A$ to be the tensor form of $(\mathfrak{B}^{\sharp})^{-1}$, that is, $T_{\phi_{\mathfrak{B}}} = (\mathfrak{B}^{\sharp})^{-1}$.

\begin{lemma}\label{lem:qpiv}
	Let $(A, \cdot_{A})$ be a perm algebra and $\mathfrak{B}$ be a nondegenerate bilinear form on $A$.
	Then $(A, \cdot_A, \mathfrak{B})$ is a quadratic perm algebra if and only if the corresponding $\phi_{\mathfrak{B}} \in A \otimes A$ is skew-symmetric and $(R, \ad)$-invariant.
\end{lemma}
\begin{proof}
	Let $a^*, b^* \in A^*$.
	Note that
	\begin{align*}
		\langle \phi_{\mathfrak{B}}, a^* \otimes b^*\rangle = \langle (\mathfrak{B}^{\sharp})^{-1}(a^*), b^*\rangle = \mathfrak{B}((\mathfrak{B}^{\sharp})^{-1}(b^*), (\mathfrak{B}^{\sharp})^{-1}(a^*)).
	\end{align*}
	It is straightforward to show that $\mathfrak{B}$ is skew-symmetric if and only if $\phi_{\mathfrak{B}}$ is skew-symmetric.
	Let $a, b, c \in A$.
	Under the skew-symmetric assumption, we have 
	\begin{align*}
		&\mathfrak{B}(a \cdot_{A} b, c) - \mathfrak{B}(a, b \cdot_{A} c - c \cdot_{A} b) =-\langle \mathfrak{B}^\sharp(c), a \cdot_{A} b \rangle - \langle \mathfrak{B}^\sharp(a), \ad(b)c\rangle \\
		&=-\langle \mathfrak{B}^\sharp(c), a \cdot_{A} b\rangle - \langle c, \ad^*(b)\mathfrak{B}^\sharp(a) \rangle =-\langle \mathfrak{B}^\sharp(c), a \cdot_{A} b\rangle - \mathfrak{B}( (\mathfrak{B}^\sharp)^{-1}( \ad^*(b) \mathfrak{B}^\sharp(a) ), c) \\
		&=-\langle \mathfrak{B}^\sharp(c), a \cdot_{A} b\rangle + \langle \mathfrak{B}^\sharp(c),  (\mathfrak{B}^\sharp)^{-1}(\ad^*(b) \mathfrak{B}^\sharp(a) ) \rangle.
	\end{align*}
	Thus, $\mathfrak{B}$ is invariant on the perm algebra $(A, \cdot_A)$ if and only 
	\begin{equation*}
		(\mathfrak{B}^\sharp)^{-1}(\ad^*(b) \mathfrak{B}^\sharp(a) ) - (\mathfrak{B}^\sharp)^{-1} (\mathfrak{B}^\sharp(a)) \cdot_{A} b = 0,
	\end{equation*}
	which, by Lemma~\ref{lem:radi}, is equivalent to the $(R, \ad)$-invariance of $\phi_{\mathfrak{B}}$.
	This completes the proof.
\end{proof}

\begin{proposition}\label{prop:pybe2rbo}
	Let $(A, \cdot_{A}, \mathfrak{B})$ be a quadratic perm algebra and $r \in A \otimes A$ such that its skew-symmetric part is $(R, \ad)$-invariant.
	Define a linear map $P: A \to A$ by 
	\begin{equation}
		P(a) = T_{r} \mathfrak{B}^\sharp(a), \;\; \forall a \in A. \label{eq:rb}
	\end{equation}
	Then $r$ is a solution of the perm Yang-Baxter equation in $(A, \cdot_{A})$ if and only if $P$ satisfies
	\begin{equation}
		P(a) \cdot_{A} P(b) = P(P(a) \cdot_{A} b + a \cdot_{A} P(b) - a \cdot_{A} T_{r - \sigma(r)}\mathfrak{B}^\sharp(b) ) \;\; \forall a, b \in A. \label{eq:prbw}
	\end{equation}
\end{proposition}
\begin{proof}
	By Lemma~\ref{lem:qpiv}, $\phi_{\mathfrak{B}}$ is skew-symmetric and $(R, \ad)$-invariant.
	Note that $r - \sigma(r)$ is skew-symmetric and $(R, \ad)$-invariant.
	For all $a , b \in A$, setting $a^* = \mathfrak{B}^\sharp(a), b^* = \mathfrak{B}^\sharp(b)$,
	we have
	\begin{align*}
		&P(a) \cdot_{A} P(b) = T_{r}(a^*) \cdot_{A} T_{r}(b^*), \\
		&P(P(a) \cdot_{A} b) = T_{r} \mathfrak{B}^\sharp( T_{r}(a^*) \cdot_{A} T_{\phi_{\mathfrak{B}}}(b^*)) \overset{\eqref{eq:adil}}{=} T_{r}( L^*(T_{r}(a^*))b^* ), \\
		&P(a \cdot_{A} P(b)) = T_{r} \mathfrak{B}^\sharp( T_{\phi_{\mathfrak{B}}}(a^*) \cdot_{A} T_{r}(b^*) ) \overset{\eqref{eq:adi}}{=} T_{r}( \ad^*(T_{r}(b^*)) a^* ), \\
		&P(a \cdot_{A} T_{r - \sigma(r)}\mathfrak{B}^\sharp(b)) =T_{r} \mathfrak{B}^\sharp( T_{\phi_{\mathfrak{B}}}(a^*) \cdot_{A} T_{r - \sigma(r)}(b^*) ) \overset{\eqref{eq:adi}}{=} T_{r}( \ad^*(T_{r-\sigma(r)}(b^*)) a^* ) \\
		&\overset{\eqref{eq:sadii}}{=} T_{r}( L^*(T_{r-\sigma(r)}(a^*)) b^* ) = T_{r}( a^* \diamond_{r - \sigma(r)} b^*),
	\end{align*}
	where $\diamond_{r-\sigma(r)}$ is defined by Eq.~\eqref{eq:apa2mul}.
	Thus, Eq.~\eqref{eq:prbw} holds if and only Eq.~\eqref{eq:pybe2rrb} holds.
	Therefore, the conclusion follows from Theorem~\ref{thm:pybe2rrb}.
\end{proof}

%%%%%%%%%%%%%%%%%%%%%%%%%%%%%%%%%%%%%%%%%%%%%%%%%%%%%%%%%%%%%%%%%%%%%%%%%%%%%%%%
%%%%%%%%%%%%%%%%%%%%%%%%%%%%%%%%%%%%%%%%%%%%%%%%%%%%%%%%%%%%%%%%%%%%%%%%%%%%%%%%
%%%%%%%%%%%%%%%%%%%%%%%%%%%%%%%%%%%%%%%%%%%%%%%%%%%%%%%%%%%%%%%%%%%%%%%%%%%%%%%%
\section{Quadratic Rota-Baxter perm algebras, triangular and factorizable perm bialgebras}\label{sec:qrflb}

In this section, we introduce the notion of quadratic Rota-Baxter perm algebras of weights. 
We show that a quadratic Rota-Baxter perm algebra of weight $0$ induces a triangular perm bialgebra, and there is a one-to-one correspondence between quadratic Rota-Baxter perm algebras of nonzero weights and factorizable perm bialgebras.

\begin{definition}
	A \textbf{quadratic Rota-Baxter perm algebra of weight $\lambda$} is a quadruple $(A, \cdot_{A}, \mathfrak{B}, P)$, where $(A, \cdot_A, \mathfrak{B})$ is a quadratic perm algebra and $P$ is a Rota-Baxter operator of weight $\lambda$ on $(A, \cdot_A)$ such that the following compatibility condition holds:
	\begin{equation}
		\mathfrak{B}(P(a), b) + \mathfrak{B}(a, P(b)) + \lambda \mathfrak{B}(a, b) = 0, \;\; \forall a, b \in A. \label{eq:qrbp}
	\end{equation}
\end{definition}

\begin{proposition}\label{prop:qrb2sqrb}
	Let $(A, \cdot_{A}, \mathfrak{B})$ be a quadratic perm algebra and $P: A \to A$ be a linear map.
	Then $(A, \cdot_{A}, \mathfrak{B}, P)$ is a quadratic Rota-Baxter perm algebra of weight $\lambda$ if and only if $(A, \cdot_{A}, -\mathfrak{B}, - (\lambda \id + P))$ is a quadratic Rota-Baxter perm algebra of weight $\lambda$.
\end{proposition}
\begin{proof}
	It is well known that $P$ is a Rota-Baxter operator of weight $\lambda$ if and only if $-(\lambda \id + P)$ is a Rota-Baxter operator of weight $\lambda$.
	For all $a, b \in A$, we have
	\begin{equation*}
		\mathfrak{B}(-(\lambda \id + P)(a), b) + \mathfrak{B}(a, -(\lambda \id + P)(b)) + \lambda\mathfrak{B}(a, b) = - \mathfrak{B}(P(a), b) - \mathfrak{B}(a, P(b)) - \lambda \mathfrak{B}(a, b).
	\end{equation*}
	Thus, $P$ satisfies Eq.~\eqref{eq:qrbp} if and only if $-(\lambda \id + P)$ satisfies Eq.~\eqref{eq:qrbp}, which completes the proof.
\end{proof}

\begin{proposition}\label{prop:smprb}
	Let $P: A \to A$ be a Rota-Baxter operator of weight $\lambda$ on a perm algebra $(A, \cdot_{A})$.
	Then $(A \ltimes_{L^*, \ad^*} A^*, \cdot, \mathfrak{B}_d, P - (\lambda \id + P)^*)$ is a quadratic Rota-Baxter perm algebra of weight $\lambda$, where the bilinear form $\mathfrak{B}_d$ on $A \oplus A^*$ is given by 
	\begin{equation*}
		\mathfrak{B}_d(a+a^*, b+b^*) = \langle a, b^*\rangle - \langle a^*, b\rangle, \;\; \forall a, b \in A, \; a^*, b^* \in A^*.
	\end{equation*} 
\end{proposition}
\begin{proof}
	It follows from a straightforward computation.
\end{proof}

\begin{lemma}\label{lem:qrbpc}
	{\rm (\cite{bai2024quasi})}
	Let $A$ be a vector space and $\mathfrak{B}$ be a nondegenerate skew-symmetric bilinear form.
	Let $r \in A \otimes A$, $\lambda \in \mathbb{K}$ and $P: A \to A$ be a linear map defined by Eq.~\eqref{eq:rb}.
	Then $r$ satisfies 
	\begin{equation}
		r - \sigma(r) = - \lambda \phi_{\mathfrak{B}} \label{eq:rsric}
	\end{equation}
	if and only if $P$ satisfies the following equation:
	\begin{equation}
		\mathfrak{B}(P(a), b) + \mathfrak{B}(a, P(b)) + \lambda \mathfrak{B}(a, b) = 0. \label{eq:qrbpc}
	\end{equation} 
\end{lemma}

\begin{proposition}
	Let $(A, \cdot_{A}, \mathfrak{B}, P)$ be a quadratic Rota-Baxter perm algebra of weight $0$.
	Set $r \in A \otimes A$ to be the tensor form of $T_r$ defined by Eq.~\eqref{eq:rb}, that is, 
	\begin{equation}
		T_{r}(a^*) = P((\mathfrak{B}^\sharp)^{-1}(a^*)), \;\; \forall a^* \in A^*. \label{eq:qrb2r}
	\end{equation}
	Then $(A, \cdot_A, \Delta_r)$ is a triangular perm bialgebra, where $\Delta_r$ is defined by Eq.~\eqref{eq:pcbd} via $r$.
\end{proposition}
\begin{proof}
	By Lemma~\ref{lem:qrbpc}, we have $r-\sigma(r) = 0$. 
	Then $r$ is a symmetric solution of the perm Yang-Baxter equation in $(A, \cdot_{A})$ by Proposition~\ref{prop:pybe2rbo}.
	Therefore, $(A, \cdot_A, \Delta_r)$ is a triangular perm bialgebra.
\end{proof}

\begin{theorem}\label{thm:qrbp2fpb}
	Let $(A, \cdot_A, \mathfrak{B}, P)$ be a quadratic Rota-Baxter perm algebra of weight $\lambda \neq 0$.
	Set $r \in A \otimes A$ to be the tensor form of $T_r$ defined by Eq.~\eqref{eq:qrb2r}.
	Then $(A, \cdot_{A}, \Delta_r)$ is a factorizable perm bialgebra, where $\Delta_r$ is defined by Eq.~\eqref{eq:pcbd} via $r$.
\end{theorem}
\begin{proof}
	By Lemma~\ref{lem:qrbpc}, we have $r - \sigma(r) = -\lambda \phi_{\mathfrak{B}}$.
	Thus, $T_{r - \sigma(r)} = -\lambda (\mathfrak{B}^{\sharp})^{-1}$ is a linear isomorphism of vector spaces, and by Lemma~\ref{lem:qpiv}, $r - \sigma(r) = -\lambda \phi_{\mathfrak{B}}$ is $(R, \ad)$-invariant.
	Moreover, since $P$ is a Rota-Baxter operator of weight $\lambda$, Eq.~\eqref{eq:prbw} holds in Proposition~\ref{prop:pybe2rbo}.
	Therefore, $r$ is solution of the perm Yang-Baxter equation in $(A, \cdot_{A})$.
	In conclusion, we show that $(A, \cdot_{A}, \Delta_r)$ is a factorizable perm bialgebra.
\end{proof}

\begin{corollary}
	Let $P: A \to A$ be a Rota-Baxter operator of weight $\lambda \neq 0$ on a perm algebra $(A, \cdot_{A})$.
	Let $\{e_1, e_2, \cdots, e_n\}$ be a basis of $A$ and $\{e_1^*, e_2^*, \cdots, e_n^*\}$ be the dual basis of $A^*$.
	Set 
	\begin{equation*}
		r = \sum_{i}((P+\lambda \id)(e_i) \otimes e_i^* + e_i^* \otimes P(e_i)).
	\end{equation*}
	Then $(A \ltimes_{L^*, \ad^*} A^*, \cdot, \Delta_r)$ is a factorizable perm bialgebra, where $\Delta_r$ is defined by Eq.~\eqref{eq:pcbd} via $r$.
\end{corollary}
\begin{proof}
	By Proposition~\ref{prop:smprb}, $(A \ltimes_{L^*, \ad^*} A^*, \cdot, \mathfrak{B}_d, P - (\lambda \id + P)^*)$ is quadratic Rota-Baxter perm algebra of weight $\lambda$.
	For all $a \in A$ and $a^* \in A^*$, we have 
	\begin{equation*}
		\mathfrak{B}_d^\sharp(a+a^*) = -a^*+a.
	\end{equation*}
	Thus, 
	\begin{align*}
		&T_{r}\mathfrak{B}_d^\sharp(a+a^*) = T_{r}(-a^*+a) = \sum_{i}\langle -a^*, (P+\lambda \id) e_i\rangle e_i^* + \sum_{i}\langle a, e_i^*\rangle P(e_i) \\
		&=  P(\sum_{i}\langle a, e_i^*\rangle e_i) -\sum_{i} \langle (P+\lambda \id)^*(a^*),  e_i\rangle e_i^*  = P(a)-(P+\lambda \id)^*(a^*),
	\end{align*}
	that is, Eq.~\eqref{eq:qrb2r} holds where $P$ is replaced by $P - (\lambda \id + P)^*$ and $\mathfrak{B}$ by $\mathfrak{B}_d$.
	Therefore, by Theorem~\ref{thm:qrbp2fpb}, $(A \ltimes_{L^*, \ad^*} A^*, \cdot, \Delta_r)$ is a factorizable perm bialgebra.
\end{proof}
%%%%%%%%%%%%%%%%%%%%%%%%%%%%%%%%%%%%%%%%%%%%%%%%%%%%%%%%%%%%%%%%%%%%%%%%%%%%%%%%

As a counterpart to Theorem~\ref{thm:qrbp2fpb}, the following theorem shows that a factorizable perm bialgebra induces a quadratic Rota-Baxter perm algebra of nonzero weight, thereby refining the one-to-one correspondence between factorizable perm bialgebras and quadratic Rota-Baxter perm algebras of nonzero weights.
\begin{theorem}\label{thm:fpb2qrbp}
	Let $(A, \cdot_A, \Delta_r)$ be a factorizable perm bialgebra.
	Let $\mathfrak{B}$ be the bilinear form on $A$ defined by 
	\begin{equation}
		\mathfrak{B}(a, b) = - \lambda \langle T_{r-\sigma(r)}^{-1}(a), b\rangle, \;\; \lambda \neq 0, \; \forall a, b \in A, \label{eq:fpb2qrp}
	\end{equation}
	and $P: A \to A$ be the linear map defined by Eq.~\eqref{eq:rb}.
	Then $(A, \cdot_A, \mathfrak{B}, P)$ is quadratic Rota-Baxter perm algebra of weight $\lambda$.
\end{theorem}
\begin{proof}
	Since $T_{r-\sigma(r)}$ is nondegenerate, by Eq.~\eqref{eq:fpb2qrp}, $\mathfrak{B}$ is a nondegenerate bilinear form on $(A, \cdot_{A})$.
	Then, by Eq.~\eqref{eq:fpb2qrp} again, we have $\mathfrak{B}^\sharp = -\lambda T_{r-\sigma(r)}^{-1}$ and thus $T_{r - \sigma(r)} = - \lambda (\mathfrak{B}^\sharp)^{-1} = -\lambda T_{\phi_{\mathfrak{B}}}$.
	Hence, $r - \sigma(r) = - \lambda \phi_{\mathfrak{B}}$.
	Note that $r-\sigma(r)$ is skew-symmetric and $(R, \ad)$-invariant. 
	By Lemma~\ref{lem:qpiv}, $(A, \cdot_{A}, \mathfrak{B})$ is a quadratic perm algebra.
	Moreover, substituting $T_{r-\sigma(r)} = -\lambda (\mathfrak{B}^\sharp)^{-1}$ into Eq.~\eqref{eq:prbw} of Proposition~\ref{prop:pybe2rbo}, we show that $P$ is a Rota-Baxter operator of weight $\lambda$.
	Finally, by Lemma~\ref{lem:qrbpc}, Eq.~\eqref{eq:qrbpc} holds.
	In conclusion, $(A, \cdot_A, \mathfrak{B}, P)$ is a quadratic Rota-Baxter perm algebra of weight $\lambda$.
\end{proof}

Combining Theorem~\ref{thm:qrbp2fpb} and \ref{thm:fpb2qrbp}, given a perm algebra $(A, \cdot_{A})$, there is a factorizable perm bialgebra $(A, \cdot_{A}, \Delta_r)$ if and only if there is a quadratic Rota-Baxter perm algebra $(A, \cdot_{A}, \mathfrak{B}, P)$ of weight $\lambda \neq 0$.
Here the mutual relation is given by 
\begin{equation*}
	r - \sigma(r) = - \lambda \phi_{\mathfrak{B}}, T_{r} = P(\mathfrak{B}^\sharp)^{-1}.
\end{equation*}
That is, there is a one-to-one correspondence between factorizable perm bialgebras and quadratic Rota-Baxter perm algebras of nonzero weight.

\begin{proposition}
	Let $(A, \cdot_{A}, \Delta_r)$ be a factorizable perm bialgebra corresponding to a quadratic Rota-Baxter perm algebra $(A, \cdot_{A}, \mathfrak{B}, P)$ of weight $\lambda \neq 0$ via Theorem~\ref{thm:qrbp2fpb} and \ref{thm:fpb2qrbp}.
	Then the factorizable perm bialgebra $(A, \cdot_{A}, \Delta_{\sigma(r)})$ corresponds to the quadratic Rota-Baxter perm algebra $(A, \cdot_{A}, -\mathfrak{B}, -(P+\lambda \id))$ of weight $\lambda$.
	In conclusion, we have the following commutative diagram.
	\begin{equation*}
		\xymatrix@C=4cm@R=1.5cm{
			\txt{$(A, \cdot_{A}, \Delta_r)$} \ar@{<->}[r]^-{{\rm Cor.}~\ref{coro:rsrf}} \ar[d]<-1ex>_-{{\rm Thm.}~\ref{thm:fpb2qrbp}} & \txt{$(A, \cdot_{A}, \Delta_{\sigma(r)})$} \ar[d]<-1ex>_-{{\rm Thm.}~\ref{thm:fpb2qrbp}} \\
			\txt{$(A, \cdot_{A}, \mathfrak{B}, P)$} \ar@{<->}[r]^-{{\rm Prop.}~\ref{prop:qrb2sqrb}} \ar[u]<-1ex>_-{{\rm Thm.}~\ref{thm:qrbp2fpb}} & \txt{$(A, \cdot_{A}, -\mathfrak{B}, -(P+\lambda \id))$} \ar[u]<-1ex>_-{{\rm Thm.}~\ref{thm:qrbp2fpb}}
		}
	\end{equation*}
\end{proposition}
\begin{proof}
	By Theorem~\ref{thm:fpb2qrbp}, $(A, \cdot_{A}, \Delta_{\sigma(r)})$ induces a quadratic Rota-Baxter perm algebra $(A, \cdot_{A}, \mathfrak{B}^\prime, P^\prime)$ of weight $\lambda$, where 
	\begin{equation*}
		\mathfrak{B}^\prime(a, b) \overset{\eqref{eq:fpb2qrp}}{=} - \lambda \langle T_{\sigma(r) - r}^{-1}\rangle = -\mathfrak{B}(a, b), 
	\end{equation*}
	and 
	\begin{align*}
		&P^\prime(a) \overset{\eqref{eq:rb}}{=} T_{\sigma(r)} (\mathfrak{B}^\prime)^\sharp(a) =  - T_{\sigma(r)} \mathfrak{B}^\sharp(a) = \lambda T_{\sigma(r)} T_{r - \sigma(r)}^{-1}(a) = \lambda (T_{r} - T_{r - \sigma(r)}) T_{r - \sigma(r)}^{-1}(a) \\
		&= -\lambda a + \lambda T_{r} T_{r - \sigma(r)}^{-1}(a) = -\lambda a - T_{r} \mathfrak{B}^\sharp(a) \overset{\eqref{eq:rb}}{=} -(\lambda \id + P)(a),
	\end{align*}
	for all $a, b \in A$.
	Thus, $(A, \cdot_{A}, \Delta_{\sigma(r)})$ induces a quadratic Rota-Baxter perm algebra $(A$, $\cdot_{A}$, $-\mathfrak{B}$, $-(\lambda \id + P))$ of weight $\lambda$.
	Conversely, by a similar argument, we show that the quadratic Rota-Baxter perm algebra $(A, \cdot_{A}, -\mathfrak{B}, -(\lambda \id + P))$ of weight $\lambda$ induces the factorizable perm bialgebra $(A, \cdot_{A}, \Delta_{\sigma(r)})$ via Theorem~\ref{thm:qrbp2fpb}.
\end{proof}

%%%%%%%%%%%%%%%%%%%%%%%%%%%%%%%%%%%%%%%%%%%%%%%%%%%%%%%%%%%%%%%%%%%%%%%%%%%%%%%%
%%%%%%%%%%%%%%%%%%%%%%%%%%%%%%%%%%%%%%%%%%%%%%%%%%%%%%%%%%%%%%%%%%%%%%%%%%%%%%%%
%%%%%%%%%%%%%%%%%%%%%%%%%%%%%%%%%%%%%%%%%%%%%%%%%%%%%%%%%%%%%%%%%%%%%%%%%%%%%%%%

%\bigskip

\smallskip

\noindent
{\bf Data availability. } No new data were created or analyzed in this study.

%%%%%%%%%%%%%%%%%%%%%%%%%%%%%%%%%%%%%%%%%%%%%%%%%%%%%%%%%%%%%%%%%%%%%%%%%%%%%%%%


\begin{thebibliography}{99}

\bibitem{aguiar2000infinitesimal}
M.~Aguiar, Infinitesimal Hopf algebras, \textit{Contemporary Mathematics} 267 (2000), 1--30.

\bibitem{aguiar2001associative}
M.~Aguiar, On the associative analog of Lie bialgebras, \textit{J. Algebra} 244 (2001), 492--532.

\bibitem{bai2010double}
C.~Bai, Double constructions of Frobenius algebras, Connes cocycles and their duality, \textit{J. Noncommutative Geom.} 4 (2010), 475--530.

\bibitem{bai2010nonabelian}
C.~Bai, L.~Guo and X.~Ni, Nonabelian generalized Lax pairs, the classical Yang-Baxter equation and post-Lie algebras, \textit{Commun. Math. Phys.} 297 (2010), 553--596.

\bibitem{bai2008left}
C.~Bai, Left-symmetric bialgebras and an analogue of the classical Yang–Baxter equation, \textit{Commun. Contemp. Math.}, 10 (2008), 221-260.


\bibitem{bai2024quasi}
C.~Bai, G.~Liu, Y.~Sheng and R.~Tang, Quasi-triangular, factorizable Leibniz bialgebras and relative Rota-Baxter operators, \textit{Forum Mathematicum}, 2024.

\bibitem{chapoton2002endofoncteur}
F.~Chapoton, Un endofoncteur de la cat{\'e}gorie des op{\'e}rades, in: \textit{Dialgebras and related operads}, Springer, 2002, pages 105--110.

\bibitem{chari1995guide}
V.~Chari and A.~Pressley, \textit{A guide to quantum groups}, Cambridge University Press, 1995.

\bibitem{drinfeld1986quantum}
V.~Drinfeld, Quantum groups, in: \textit{Proc. Int. Congr. Math.} 1 (1986), 798--820.

\bibitem{drinfeld1983hamiltonian}
V.~G. Drinfeld, Hamiltonian structures on Lie groups, Lie bialgebras and the geometric meaning of the classical Yang-Baxter equations, \textit{Soviet Math. Dokl.} 27 (1983), 68--71.

\bibitem{joni1979coalgebras}
S.~Joni and G.-C. Rota, Coalgebras and bialgebras in combinatorics, \textit{Stud. Appl. Math.} 61 (1979), 93--139.

\bibitem{kassel2012quantum}
C.~Kassel, \textit{Quantum groups}, Springer Science \& Business Media, 2012.

\bibitem{knapp2002lie}
A.~W. Knapp, \textit{Lie Groups Beyond an Introduction}, Springer Science \& Business Media, 2002.

\bibitem{kock2004frobenius}
J.~Kock, \textit{Frobenius algebras and 2D topological quantum field theories}, Cambridge University Press, 2004.

\bibitem{lang2023factorizable}
H.~Lang and Y.~Sheng, Factorizable Lie bialgebras, quadratic Rota-Baxter Lie algebras and Rota-Baxter Lie bialgebras, \textit{Commun. Math. Phys.} 397 (2023), 763--791.

\bibitem{lauda2008open}
A.~D. Lauda and H.~Pfeiffer, Open--closed strings: Two-dimensional extended TQFTs and Frobenius algebras, \textit{Topology Appl.} 155 (2008), 623--666.

\bibitem{lin2025infinite}
Y.~Lin, P.~Zhou and C.~Bai, Infinite-dimensional {Lie} bialgebras via affinization of perm bialgebras and pre-Lie bialgebras, \textit{J. Algebra} 663 (2025), 210--258.

\bibitem{wang2024quasi}
Y.~Wang, C.~Bai, J.~Liu, Quasi-triangular pre-Lie bialgebras, factorizable pre-Lie bialgebras and Rota-Baxter pre-Lie algebras, \textit{J. Geom. Phys.} 199 (2024), 105146. 


\bibitem{reshetikhin1988quantum}
N.~Y. Reshetikhin and M.~A. Semenov-Tian-Shansky, Quantum R-matrices and factorization problems, \textit{J. Geom. Phys.} 5 (1988), 533--550.

\bibitem{semenov2003integrable}
M.~A. Semenov-Tian-Shansky, Integrable systems and factorization problems, in: \textit{Factorization and Integrable Systems: Summer School in Faro, Portugal, September 2000}, Springer, 2003, pages 155--218.

\bibitem{semenov1983classical}
M.~Semenov-Tyan-Shanskii, What is a classical $r$-matrix?, \textit{Funct. Anal. Appl.} 17 (1983), 259--272.

\bibitem{sheng2023quasi}
Y.~Sheng and Y.~Wang, Quasi-triangular and factorizable antisymmetric infinitesimal bialgebras, \textit{J. Algebra} 628 (2023), 415--433.

\bibitem{tang2022leibniz}
R.~Tang, Y.~Sheng, Leibniz bialgebras, relative Rota–Baxter operators, and the classical Leibniz Yang–Baxter equation, \textit{J. Noncommut. Geom.} 16 (2022), 1179-1211.

\bibitem{zhelyabin1997jordan}
V.~N. Zhelyabin, Jordan bialgebras and their connection with {Lie} bialgebras, \textit{Algebra Log.} 36 (1997), 3--25.


\end{thebibliography}
\end{document}